\documentclass[reqno]{amsart}
\usepackage{amssymb,amsmath}%
\usepackage{cite}
\usepackage{ifpdf}
\usepackage{enumitem}
\ifpdf
\usepackage[hyperindex]{hyperref}
\else
\expandafter \ifx \csname dvipdfm \endcsname \relax
\usepackage[hypertex,hyperindex]{hyperref}
\else
\usepackage[dvipdfm,hyperindex]{hyperref}
\fi
\fi

\newcommand{\mymod}[3]{#1 \equiv #2 \kern -0.5em \pmod{#3}}
\newcommand{\mynotmod}[3]{#1 \not \equiv #2 \kern -0.6em \pmod{#3}}

\numberwithin{equation}{section}
\newtheorem{theorem}{Theorem}[section]

\begin{document}
\title[The Third Order Jacobsthal Octonions]{The Third Order Jacobsthal Octonions: Some Combinatorial Properties}

\author[G. Cerda-Morales]{Gamaliel Cerda-Morales}  

\address{Gamaliel Cerda-Morales \newline
 Instituto de Matem\'aticas, Pontificia Universidad Cat\'olica de Valpara\'iso, Blanco Viel 596, Valpara\'iso, Chile.}
\email{gamaliel.cerda.m@mail.pucv.cl}



\begin{abstract}
Various families of octonion number sequences (such as Fibonacci octonion, Pell octonion and Jacobsthal octonion) have been established by a number of authors in many different ways. In addition, formulas and identities involving these number sequences have been presented. In this paper, we aim at establishing new classes of octonion numbers associated with the third order Jacobsthal and third order Jacobsthal-Lucas numbers. We introduce the third order Jacobsthal octonions and the third order Jacobsthal-Lucas octonions and give some of their properties. We derive the relations
between third order Jacobsthal octonions and third order Jacobsthal-Lucas octonions.
\end{abstract}

\maketitle

\section{Introduction}
In recent years, the topic of number sequences in real normed division algebras has attracted the attention of several researchers. It is worth noticing that there are exactly four real normed division algebras: real numbers ($\Bbb{R}$),
complex numbers ($\Bbb{C}$), quaternions ($\Bbb{H}$) and octonions ($\Bbb{O}$). Baez \cite{Bae} gives a comprehensive discussion of these algebras.

The real quaternion algebra $\Bbb{H}=\{q=q_{r}+q_{i}\textbf{i}+q_{j}\textbf{j}+q_{k}\textbf{k}:\ q_{r},q_{s}\in \Bbb{R},\ s=i,j,k\}$ is a 4-dimensional $\Bbb{R}$-vector space with basis $\{\textbf{1}\simeq e_{0},\textbf{i}\simeq e_{1},\textbf{j}\simeq e_{2},\textbf{k}\simeq e_{3}\}$ satisfying multiplication rules $q_{r}\textbf{1}=q_{r}$, $e_{1}e_{2}=-e_{2}e_{1}=e_{3}$, $e_{2}e_{3}=-e_{3}e_{2}=e_{1}$ and $e_{3}e_{1}=-e_{1}e_{3}=e_{2}$. The real octonion algebra denoted by $\Bbb{O}$ is an 8-dimensional real linear space with basis
\begin{equation}\label{eq:0}
\{e_{0}=\textbf{1}, e_{1}=\textbf{i}, e_{2}=\textbf{j}, e_{3}=\textbf{k}, e_{4}=\textbf{e}, e_{5}=\textbf{ie}, e_{6}=\textbf{je}, e_{7}=\textbf{ke}\},
\end{equation}
where $e_{0}\cdot e_{s}=e_{s}$ ($s=1,...,7$) and $q_{r}e_{0}=q_{r}$ ($q_{r}\in \Bbb{R}$).
The space $\Bbb{O}$ becomes an algebra via multiplication rules listed in the table \ref{table:1}, see \cite{Ta}.

\begin{table}[ht] 
\caption{The multiplication table for the basis of $\Bbb{O}$.} 
\centering      
\begin{tabular}{llllllll}
\hline
$\times $ & $e_{1}$ & $e_{2}$ & $e_{3}$ & $e_{4}$ & $e_{5}$ & $e_{6}$ & $%
e_{7}$ \\ \hline
$e_{1}$ & $-1$ & $e_{3}$ & $-e_{2}$ & $e_{5}$ & $-e_{4}$ & $-e_{7}$ & $e_{6}$
\\ 
$e_{2}$ & $-e_{3}$ & $-1$ & $e_{1}$ & $e_{6}$ & $e_{7}$ & $-e_{4}$ & $-e_{5}$
\\ 
$e_{3}$ & $e_{2}$ & $-e_{1}$ & $-1$ & $e_{7}$ & $-e_{6}$ & $e_{5}$ & $-e_{4}$
\\ 
$e_{4}$ & $-e_{5}$ & $-e_{6}$ & $-e_{7}$ & $-1$ & $e_{1}$ & $e_{2}$ & $e_{3}$
\\ 
$e_{5}$ & $e_{4}$ & $-e_{7}$ & $e_{6}$ & $-e_{1}$ & $-1$ & $-e_{3}$ & $e_{2}$
\\ 
$e_{6}$ & $e_{7}$ & $e_{4}$ & $-e_{5}$ & $-e_{2}$ & $e_{3}$ & $-1$ & $-e_{1}$
\\ 
$e_{7}$ & $-e_{6}$ & $e_{5}$ & $e_{4}$ & $-e_{3}$ & $-e_{2}$ & $e_{1}$ & $-1$
\\ \hline
\end{tabular}
\label{table:1}  
\end{table}

A variety of new results on Fibonacci-like quaternion and octonion numbers can be found in several papers \cite{Cer,Cim1,Cim2,Hal1,Hal2,Hor1,Hor2,Iye,Ke-Ak,Szy-Wl}. The origin of the topic of number sequences in division algebra can be traced back to the works by Horadam in \cite{Hor1} and by Iyer in \cite{Iye}. Horadam \cite{Hor1} defined the quaternions
with the classic Fibonacci and Lucas number components as
\[
QF_{n}=F_{n}+F_{n+1}\textbf{i}+F_{n+2}\textbf{j}+F_{n+3}\textbf{k}
\]
and
\[
QL_{n}=L_{n}+L_{n+1}\textbf{i}+L_{n+2}\textbf{j}+L_{n+3}\textbf{k},
\]
respectively, where $F_{n}$ and $L_{n}$ are the $n$-th classic Fibonacci and Lucas numbers, respectively, and the author studied the properties of these quaternions. Several interesting and useful extensions of many of the familiar quaternion
numbers (such as the Fibonacci and Lucas quaternions \cite{Aky,Hal1,Hor1}, Pell quaternion \cite{Ca,Cim1} and Jacobsthal quaternions \cite{Szy-Wl}) have been considered by several authors.

In this paper, we define two families of the octonions, where the coefficients in the terms of the octonions are determined by the third order Jacobsthal and third order Jacobsthal-Lucas numbers. These two families of the octonions are called as the third order Jacobsthal and third order Jacobsthal-Lucas octonions, respectively. We mention some of their properties, and apply them to the study of some identities and formulas of the third order Jacobsthal and third order Jacobsthal-Lucas octonions. 

Here, our approach for obtaining some fundamental properties and characteristics of third order Jacobsthal and third order Jacobsthal-Lucas octonions is to apply the properties of the third order Jacobsthal and third order Jacobsthal-Lucas numbers introduced by Cook and Bacon \cite{Cook-Bac}. This approach was originally proposed by Horadam and Iyer in the articles \cite{Hor1,Iye} for Fibonacci quaternions. The methods used by Horadam and Iyer in that papers have been recently applied to the other familiar octonion numbers by several authors \cite{Ak,Ca,Cim2,Ke-Ak}. 

This paper has three main sections. In Section \ref{sect:2}, we provide the basic definitions of the octonions and the third order Jacobsthal and third order Jacobsthal-Lucas numbers. Section \ref{sect:3} is devoted to introducing third order Jacobsthal and third order Jacobsthal-Lucas octonions, and then to obtaining some fundamental properties and characteristics of these numbers. 

\section{Preliminaries}\label{sect:2}
The Jacobsthal numbers have many interesting properties and applications in many fields of science (see, e.g., \cite{Ba}). The Jacobsthal numbers $J_{n}$ are defined by the recurrence relation
\begin{equation}\label{e1}
J_{0}=0,\ J_{1}=1,\ J_{n+1}=J_{n}+2J_{n-1},\ n\geq1.
\end{equation}
Another important sequence is the Jacobsthal-Lucas sequence. This sequence is defined by the recurrence relation
\begin{equation}\label{ec1}
j_{0}=2,\ j_{1}=1,\ j_{n+1}=j_{n}+2j_{n-1},\ n\geq1.
\end{equation}
(see, \cite{Hor3}).

In \cite{Cook-Bac} the Jacobsthal recurrence relation is extended to higher order recurrence relations and the basic list of identities provided by A. F. Horadam \cite{Hor3} is expanded and extended to several identities for some of the higher order cases. For example, third order Jacobsthal numbers, $\{J_{n}^{(3)}\}_{n\geq0}$, and third order Jacobsthal-Lucas numbers, $\{j_{n}^{(3)}\}_{n\geq0}$, are defined by
\begin{equation}\label{e2}
J_{n+3}^{(3)}=J_{n+2}^{(3)}+J_{n+1}^{(3)}+2J_{n}^{(3)},\ J_{0}^{(3)}=0,\ J_{1}^{(3)}=J_{2}^{(3)}=1,\ n\geq0,
\end{equation}
and 
\begin{equation}\label{e3}
j_{n+3}^{(3)}=j_{n+2}^{(3)}+j_{n+1}^{(3)}+2j_{n}^{(3)},\ j_{0}^{(3)}=2,\ j_{1}^{(3)}=1,\ j_{2}^{(3)}=5,\ n\geq0,
\end{equation}
respectively.

The following properties given for third order Jacobsthal numbers and third order Jacobsthal-Lucas numbers play important roles in this paper (see \cite{Cer,Cook-Bac}). 
\begin{equation}\label{e4}
3J_{n}^{(3)}+j_{n}^{(3)}=2^{n+1},
\end{equation}
\begin{equation}\label{e5}
j_{n}^{(3)}-3J_{n}^{(3)}=2j_{n-3}^{(3)},
\end{equation}
\begin{equation}\label{ec5}
J_{n+2}^{(3)}-4J_{n}^{(3)}=\left\{ 
\begin{array}{ccc}
-2 & \textrm{if} & \mymod{n}{1}{3} \\ 
1 & \textrm{if} & \mynotmod{n}{1}{3}%
\end{array}%
\right. ,
\end{equation}
\begin{equation}\label{e6}
j_{n}^{(3)}-4J_{n}^{(3)}=\left\{ 
\begin{array}{ccc}
2 & \textrm{if} & \mymod{n}{0}{3} \\ 
-3 & \textrm{if} & \mymod{n}{1}{3}\\ 
1 & \textrm{if} & \mymod{n}{2}{3}%
\end{array}%
\right. ,
\end{equation}
\begin{equation}\label{e7}
j_{n+1}^{(3)}+j_{n}^{(3)}=3J_{n+2}^{(3)},
\end{equation}
\begin{equation}\label{e8}
j_{n}^{(3)}-J_{n+2}^{(3)}=\left\{ 
\begin{array}{ccc}
1 & \textrm{if} & \mymod{n}{0}{3} \\ 
-1 & \textrm{if} & \mymod{n}{1}{3} \\ 
0 & \textrm{if} & \mymod{n}{2}{3}%
\end{array}%
\right. ,
\end{equation}
\begin{equation}\label{e9}
\left( j_{n-3}^{(3)}\right) ^{2}+3J_{n}^{(3)}j_{n}^{(3)}=4^{n},
\end{equation}
\begin{equation}\label{e10}
\sum\limits_{k=0}^{n}J_{k}^{(3)}=\left\{ 
\begin{array}{ccc}
J_{n+1}^{(3)} & \textrm{if} & \mynotmod{n}{0}{3} \\ 
J_{n+1}^{(3)}-1 & \textrm{if} & \mymod{n}{0}{3}%
\end{array}%
\right. ,
\end{equation}
\begin{equation}\label{e11}
\sum\limits_{k=0}^{n}j_{k}^{(3)}=\left\{ 
\begin{array}{ccc}
j_{n+1}^{(3)}-2 & \textrm{if} & \mynotmod{n}{0}{3} \\ 
j_{n+1}^{(3)}+1 & \textrm{if} & \mymod{n}{0}{3}%
\end{array}%
\right. 
\end{equation}
and
\begin{equation}\label{e12}
\left( j_{n}^{(3)}\right) ^{2}-9\left( J_{n}^{(3)}\right)^{2}=2^{n+2}j_{n-3}^{(3)}.
\end{equation}

Using standard techniques for solving recurrence relations, the auxiliary equation, and its roots are given by 
$$x^{3}-x^{2}-x-2=0;\ x = 2,\ \textrm{and}\ x=\frac{-1\pm i\sqrt{3}}{2}.$$ 

Note that the latter two are the complex conjugate cube roots of unity. Call them $\omega_{1}$ and $\omega_{2}$, respectively. Thus the Binet formulas can be written as
\begin{equation}\label{b1}
J_{n}^{(3)}=\frac{2}{7}2^{n}-\frac{3+2i\sqrt{3}}{21}\omega_{1}^{n}-\frac{3-2i\sqrt{3}}{21}\omega_{2}^{n}
\end{equation}
and
\begin{equation}\label{b2}
j_{n}^{(3)}=\frac{8}{7}2^{n}+\frac{3+2i\sqrt{3}}{7}\omega_{1}^{n}+\frac{3-2i\sqrt{3}}{7}\omega_{2}^{n},
\end{equation}
respectively.

In the following we will study the important properties of the octonions. We refer to \cite{Bae} for a detailed analysis of the properties of the octonions $p=\sum_{s=0}^{7}a_{s}e_{s}$ and $q=\sum_{s=0}^{7}b_{s}e_{s}$ where the coefficients $a_{s}$ and $b_{s}$, $s=0,1,...,7$, are real. We recall here only the following facts
\begin{itemize}[noitemsep]
\item The sum and subtract of $p$ and $q$ is defined as 
\begin{equation}\label{s1}
p\pm q=\sum_{s=0}^{7}(a_{s}\pm b_{s})e_{s},
\end{equation}
where $p\in \Bbb{O}$ can be written as $p=R_{p}+I_{p}$, and $R_{p}=a_{0}$ and $\sum_{s=1}^{7}a_{s}e_{s}$ are called the real and imaginary parts, respectively.
\item The conjugate of $p$ is defined by 
\begin{equation}\label{s2}
\overline{p}=R_{p}-I_{p}=a_{0}-\sum_{s=1}^{7}a_{s}e_{s}
\end{equation}
and this operation satisfies $\overline{\overline{p}}=p$, $\overline{p+q}=\overline{p}+\overline{q}$ and $\overline{p \cdot q}=\overline{q}\cdot \overline{p}$, for all $p,q\in \Bbb{O}$.
\item The norm of an octonion, which agrees with the standard Euclidean norm on $\Bbb{R}^{8}$ is defined as 
\begin{equation}\label{s3}
Nr^{2}(p)=\overline{p}\cdot p=p\cdot \overline{p}=\sum_{s=0}^{7}a_{s}^{2}.
\end{equation}
\item The inverse of $p\neq 0$ is given by $p^{-1}=\frac{\overline{p}}{Nr^{2}(p)}$. From the above two definitions it is deduced that 
\begin{equation}\label{s4}
Nr^{2}(p\cdot q)=Nr^{2}(p)Nr^{2}(q)\ \textrm{and}\ (p\cdot q)^{-1}=q^{-1}\cdot p^{-1}.
\end{equation}
\item $\Bbb{O}$ is non-commutative and non-associative but it is alternative, in other words 
\begin{equation}\label{s5}
\begin{aligned}
p\cdot(p\cdot q)&=p^{2}\cdot q,\\
(p\cdot q)\cdot q&=p\cdot q^{2},\\
(p\cdot q)\cdot p&=p\cdot (q\cdot p)=p\cdot q\cdot p,
\end{aligned}
\end{equation}
where $\cdot$ denotes the product in the octonion algebra $\Bbb{O}$.
\end{itemize}

\section{Third order Jacobsthal Octonions and third order Jacobsthal-Lucas Octonions}\label{sect:3}
In this section, we define new kinds of sequences of octonion number called as third order Jacobsthal octonions and third order Jacobsthal-Lucas octonions. We study some properties of these octonions. We obtain various results for these classes of octonion numbers included recurrence relations, summation formulas, Binet's formulas and generating functions.

In \cite{Cer}, the author introduced the so-called third order Jacobsthal quaternions, which are a new class of quaternion sequences. They are defined by
\begin{equation}\label{eq:1}
JQ_{n}^{(3)}=\sum_{s=0}^{3}J_{n+s}^{(3)}e_{s}=J_{n}^{(3)}+\sum_{s=1}^{3}J_{n+s}^{(3)}e_{s},\ (J_{n}^{(3)}\textbf{1}=J_{n}^{(3)})
\end{equation}
where $J_{n}^{(3)}$ is the $n$-th third order Jacobsthal number, $e_{1}^{2}=e_{2}^{2}=e_{3}^{2}=-\textbf{1}$ and $e_{1}e_{2}e_{3}=-\textbf{1}$.

We now consider the usual third order Jacobsthal and third order Jacobsthal-Lucas numbers, and based on the definition (\ref{eq:1}) we give definitions of new kinds of octonion numbers, which we call the third order Jacobsthal octonions and third order Jacobsthal-Lucas octonions. In this paper, we define the $n$-th third order Jacobsthal octonion and 
third order Jacobsthal-Lucas octonion numbers, respectively, by the following recurrence relations
\begin{equation}
\begin{aligned}
JO_{n}^{(3)}&=J_{n}^{(3)}+\sum_{s=1}^{7}J_{n+s}^{(3)}e_{s},\ n\geq 0\\
&=J_{n}^{(3)}+J_{n+1}^{(3)}e_{1}+J_{n+2}^{(3)}e_{2}+J_{n+3}^{(3)}e_{3}\\
&\ \ +J_{n+4}^{(3)}e_{4}+J_{n+5}^{(3)}e_{5}+J_{n+6}^{(3)}e_{6}+J_{n+7}^{(3)}e_{7}
\end{aligned}\label{eq:2}
\end{equation}
and
\begin{equation}
\begin{aligned}
jO_{n}^{(3)}&=j_{n}^{(3)}+\sum_{s=1}^{7}j_{n+s}^{(3)}e_{s},\ n\geq 0\\
&=j_{n}^{(3)}+j_{n+1}^{(3)}e_{1}+j_{n+2}^{(3)}e_{2}+j_{n+3}^{(3)}e_{3}\\
&\ \ +j_{n+4}^{(3)}e_{4}+j_{n+5}^{(3)}e_{5}+j_{n+6}^{(3)}e_{6}+j_{n+7}^{(3)}e_{7},
\end{aligned}\label{eq:3}
\end{equation}
where $J_{n}^{(3)}$ and $j_{n}^{(3)}$ are the $n$-th third order Jacobsthal number and third order Jacobsthal-Lucas number, respectively. Here $\{e_{s}:\ s=0,1,...,7\}$ satisfies the multiplication rule given in the Table \ref{table:1}.

The equalities in (\ref{s1}) gives 
\begin{equation}\label{s6}
JO_{n}^{(3)}\pm jO_{n}^{(3)}=\sum_{s=0}^{7}(J_{n+s}^{(3)}\pm j_{n+s}^{(3)})e_{s}.
\end{equation}
From (\ref{s2}), (\ref{s3}), (\ref{eq:2}) and (\ref{eq:3}) an easy computation gives 
\begin{equation}\label{s7}
\overline{JO_{n}^{(3)}}=J_{n}^{(3)}-\sum_{s=1}^{7}J_{n+s}^{(3)}e_{s},\ \overline{jO_{n}^{(3)}}=j_{n}^{(3)}-\sum_{s=1}^{7}j_{n+s}^{(3)}e_{s},
\end{equation}
\begin{equation}\label{s8}
Nr^{2}(JO_{n}^{(3)})=\sum_{s=0}^{7}(J_{n+s}^{(3)})^{2}\ \textrm{and}\ Nr^{2}(jO_{n}^{(3)})=\sum_{s=0}^{7}(j_{n+s}^{(3)})^{2}.
\end{equation}

By some elementary calculations we find the following recurrence relations for the third order Jacobsthal and third order Jacobsthal-Lucas octonions from (\ref{eq:2}), (\ref{eq:3}), (\ref{e2}), (\ref{e3}) and (\ref{s6}):
\begin{equation}
\begin{aligned}
JO_{n+1}^{(3)}+JO_{n}^{(3)}+2JO_{n-1}^{(3)}&=\sum_{s=0}^{7}(J_{n+s+1}^{(3)}+J_{n+s}^{(3)}+2J_{n+s-1}^{(3)})e_{s}\\
&=J_{n+2}^{(3)}+\sum_{s=1}^{7}J_{n+s+2}^{(3)}e_{s}\\
&=JO_{n+2}^{(3)}
\end{aligned} \label{equ:3}
\end{equation}
and similarly $jO_{n+2}^{(3)}=jO_{n+1}^{(3)}+jO_{n}^{(3)}+2jO_{n-1}^{(3)}$, for $n\geq1$.

Now, we will state Binet's formulas for the third order Jacobsthal and third order Jacobsthal-Lucas octonions. Repeated use of (\ref{e2}) in (\ref{eq:2}) enables one to write for $\underline{\alpha}=\sum_{s=0}^{7}2^{s}e_{s}$, $\underline{\omega_{1}}=\sum_{s=0}^{7}\omega_{1}^{s}e_{s}$ and $\underline{\omega_{2}}=\sum_{s=0}^{7}\omega_{2}^{s}e_{s}$ 
\begin{equation}
\begin{aligned}
JO_{n}^{(3)}&=J_{n}^{(3)}+\sum_{s=1}^{7}J_{n+s}^{(3)}e_{s}\\
&=\sum_{s=1}^{7}\left(\frac{2}{7}2^{n+s}-\frac{3+2i\sqrt{3}}{21}\omega_{1}^{n+s}-\frac{3-2i\sqrt{3}}{21}\omega_{2}^{n+s}\right)e_{s}\\
&=\frac{2}{7}\underline{\alpha}2^{n}-\frac{3+2i\sqrt{3}}{21}\underline{\omega_{1}}\omega_{1}^{n}-\frac{3-2i\sqrt{3}}{21}\underline{\omega_{2}}\omega_{2}^{n}
\end{aligned} \label{equ:4}
\end{equation}
and similarly making use of (\ref{e3}) in (\ref{eq:3}) yields
\begin{equation}
\begin{aligned}
jO_{n}^{(3)}&=j_{n}^{(3)}+\sum_{s=1}^{7}j_{n+s}^{(3)}e_{s}\\
&=\sum_{s=1}^{7}\left(\frac{8}{7}2^{n+s}+\frac{3+2i\sqrt{3}}{7}\omega_{1}^{n+s}+\frac{3-2i\sqrt{3}}{7}\omega_{2}^{n+s}\right)e_{s}\\
&=\frac{8}{7}\underline{\alpha}2^{n}+\frac{3+2i\sqrt{3}}{7}\underline{\omega_{1}}\omega_{1}^{n}+\frac{3-2i\sqrt{3}}{7}\underline{\omega_{2}}\omega_{2}^{n}.
\end{aligned} \label{equ:5}
\end{equation}
The formulas in (\ref{equ:4}) and (\ref{equ:5}) are called as Binet's formulas for the third order Jacobsthal and third order Jacobsthal-Lucas octonions, respectively. The recurrence relations for the third order Jacobsthal octonions and the norm of the $n$-th third order Jacobsthal octonion are expressed in the following theorem.

\begin{theorem}\label{th:1}
For $n\geq 0$, we have the following identities:
\begin{equation}\label{t1}
JO_{n+2}^{(3)}+JO_{n+1}^{(3)}+JO_{n}^{(3)}=2^{n+1}\underline{\alpha},
\end{equation}
\begin{equation}\label{t2}
JO_{n+2}^{(3)}-4JO_{n}^{(3)}=\left\{ 
\begin{array}{ccc}
\left(
\begin{array}{c}
1+e_{2}+e_{3}+e_{5}+e_{6}\\
-2(e_{1}+e_{4}+e_{7})
\end{array}%
\right)  & \textrm{if} & \mymod{n}{0}{3} \\ 
\left(
\begin{array}{c}
e_{1}+e_{2}+e_{4}+e_{5}+e_{7}\\
-2(1+e_{3}+e_{6})
\end{array}%
\right) & \textrm{if} & \mymod{n}{1}{3}\\
\left(
\begin{array}{c}
1+e_{1}+e_{3}+e_{4}+e_{6}\\
+e_{7}-2(e_{2}+e_{5})
\end{array}%
\right) & \textrm{if} & \mymod{n}{2}{3} 
\end{array}%
\right. ,
\end{equation}
\begin{equation}\label{t3}
Nr^{2}(JO_{n}^{(3)})=\frac{1}{49}\cdot \left\{ 
\begin{array}{ccc}
87380\cdot 2^{2n}+1024\cdot 2^{n}+41 & \textrm{if} & \mymod{n}{0}{3} \\ 
87380\cdot 2^{2n}+4\cdot 2^{n}+38 & \textrm{if} & \mymod{n}{1}{3} \\ 
87380\cdot 2^{2n}-1028\cdot 2^{n}+33 & \textrm{if} & \mymod{n}{2}{3}%
\end{array}%
\right. ,
\end{equation}
where $\underline{\alpha}=\sum_{s=0}^{7}2^{s}e_{s}$.
\end{theorem}
\begin{proof}
Consider (\ref{s1}) and (\ref{eq:2}) we can write
\begin{equation}\label{p1}
JO_{n+2}^{(3)}+JO_{n+1}^{(3)}+JO_{n}^{(3)}=\sum_{s=0}^{7}(J_{n+s+2}^{(3)}+J_{n+s+1}^{(3)}+J_{n+s}^{(3)})e_{s}.
\end{equation}

Using the identity $J_{n+2}^{(3)}+J_{n+1}^{(3)}+J_{n}^{(3)}=2^{n+1}$ in (\ref{p1}), the above sum can be calculated as
\[ 
JO_{n+2}^{(3)}+JO_{n+1}^{(3)}+JO_{n}^{(3)}=\sum_{s=0}^{7}2^{n+s+1}e_{s}, 
\]
which can be simplified as $JO_{n+2}^{(3)}+JO_{n+1}^{(3)}+JO_{n}^{(3)}=2^{n+1}\underline{\alpha}$, where $\underline{\alpha}=\sum_{s=0}^{7}2^{s}e_{s}$. Now, using (\ref{ec5}) and (\ref{eq:2}) we can write $JO_{n+2}^{(3)}-4JO_{n}^{(3)}=\sum_{s=0}^{7}(J_{n+s+2}^{(3)}-4J_{n+s}^{(3)})e_{s}$, then
\begin{align*}
JO_{n+2}^{(3)}-4JO_{n}^{(3)}&=\sum_{s=0}^{7}(J_{n+s+2}^{(3)}-4J_{n+s}^{(3)})e_{s}\\
&=J_{n+2}^{(3)}-4J_{n}^{(3)}+(J_{n+3}^{(3)}-4J_{n+1}^{(3)})e_{1}+\cdots +(J_{n+9}^{(3)}-4J_{n+7}^{(3)})e_{7}\\
&=e_{1}+e_{2}+e_{4}+e_{5}+e_{7}-2(1+e_{3}+e_{6})
\end{align*}
if $\mymod{n}{1}{3}$. Similarly in the other cases, this proves (\ref{t2}). Finally, observing that $Nr^{2}(JO_{n}^{(3)})=\sum_{s=0}^{7}(J_{n+s}^{(3)})^{2}$ from the Binet formula (\ref{b1}) we have
\begin{align*}
Nr^{2}(JO_{n}^{(3)})&=(J_{n}^{(3)})^{2}+(J_{n+1}^{(3)})^{2}+\cdots +(J_{n+7}^{(3)})^{2}\\
&=\frac{1}{49}\left((2^{n+1}-(a\omega_{1}^{n}+b\omega_{2}^{n}))^{2}+\cdots +((2^{n+8}-(a\omega_{1}^{n+7}+b\omega_{2}^{n+7}))^{2}\right),
\end{align*}
where $a=1+\frac{2i\sqrt{3}}{3}$ and $b=1-\frac{2i\sqrt{3}}{3}$. It is easy to see that,
\begin{equation}\label{h5}
a\omega_{1}^{n}+b\omega_{2}^{n}=\left\{ 
\begin{array}{ccc}
2 & \textrm{if} & \mymod{n}{0}{3} \\ 
-3& \textrm{if} & \mymod{n}{1}{3} \\ 
1& \textrm{if} & \mymod{n}{2}{3}%
\end{array}%
\right. ,
\end{equation}
because $\omega_{1}$ and $\omega_{2}$ are the complex conjugate cube roots of unity (i.e. $\omega_{1}^{3}=\omega_{2}^{3}=1$). Then, if we consider first $n\equiv 0(\textrm{mod}\ 3)$, we obtain
\begin{align*}
Nr^{2}(JO_{n}^{(3)}))&=\frac{1}{49}\left(\begin{array}{c} 
\left(2^{n+1}-2\right)^{2}+\left(2^{n+2}+3\right)^{2}+\left(2^{n+3}-1\right)^{2}+\left(2^{n+4}-2\right)^{2}\\
+\left(2^{n+5}+3\right)^{2}+\left(2^{n+6}-1\right)^{2}+\left(2^{n+7}-2\right)^{2}+\left(2^{n+8}+3\right)^{2}
\end{array}\right)\\
&=\frac{1}{49}\left(21845\cdot 2^{2n+2}+2^{n+10}+41\right).
\end{align*}
The other identities are clear from equation (\ref{h5}).
\end{proof}

The recurrence relations for the third order Jacobsthal-Lucas octonions and the norm of the $n$-th third order Jacobsthal-Lucas octonion are given in the following theorem.

\begin{theorem}\label{th:2}
Let $n\geq 0$ be integer. Then,
\begin{equation}\label{t4}
jO_{n+2}^{(3)}+jO_{n+1}^{(3)}+jO_{n}^{(3)}=2^{n+3}\underline{\alpha},
\end{equation}
\begin{equation}\label{t5}
jO_{n+2}^{(3)}-4jO_{n}^{(3)}=\left\{ 
\begin{array}{ccc}
\left(
\begin{array}{c}
-3(1+e_{2}+e_{3}+e_{5}+e_{6})\\
+6(e_{1}+e_{4}+e_{7})
\end{array}%
\right)  & \textrm{if} & \mymod{n}{0}{3} \\ 
\left(
\begin{array}{c}
-3(e_{1}+e_{2}+e_{4}+e_{5}+e_{7})\\
+6(1+e_{3}+e_{6})
\end{array}%
\right) & \textrm{if} & \mymod{n}{1}{3}\\
\left(
\begin{array}{c}
-3(1+e_{1}+e_{3}+e_{4}+e_{6})\\
-3e_{7}+6(e_{2}+e_{5})
\end{array}%
\right) & \textrm{if} & \mymod{n}{2}{3} 
\end{array}%
\right. ,
\end{equation}
\begin{equation}\label{t6}
Nr^{2}(jO_{n}^{(3)})=\frac{1}{49}\cdot \left\{ 
\begin{array}{ccc}
21845\cdot 2^{2n+6}-12288\cdot 2^{n}+369 & \textrm{if} & \mymod{n}{0}{3} \\ 
21845\cdot 2^{2n+6}-48\cdot 2^{n}+342 & \textrm{if} & \mymod{n}{1}{3} \\ 
21845\cdot 2^{2n+6}+12336\cdot 2^{n}+297 & \textrm{if} & \mymod{n}{2}{3}%
\end{array}%
\right. ,
\end{equation}
where $\underline{\alpha}=\sum_{s=0}^{7}2^{s}e_{s}$.
\end{theorem}

The proofs of the identities (\ref{t4})--(\ref{t6}) of this theorem are similar to the proofs of the identities (\ref{t1})--(\ref{t3}) of Theorem \ref{th:1}, respectively, and are omitted here.

The following theorem deals with two relations between the third order Jacobsthal and third order Jacobsthal-Lucas octonions.
\begin{theorem}\label{th:3}
Let $n\geq 0$ be integer. Then,
\begin{equation}\label{t7}
jO_{n+3}^{(3)}-3JO_{n+3}^{(3)}=2jO_{n}^{(3)},
\end{equation}
\begin{equation}\label{t8}
jO_{n}^{(3)}+jO_{n+1}^{(3)}=3JO_{n+2}^{(3)},
\end{equation}
\begin{equation}\label{t9}
jO_{n}^{(3)}-JO_{n+2}^{(3)}=\left\{ 
\begin{array}{ccc}
1-e_{1}+e_{3}-e_{4}+e_{6}-e_{7} & \textrm{if} & \mymod{n}{0}{3} \\ 
-1+e_{2}-e_{3}+e_{5}-e_{6} & \textrm{if} & \mymod{n}{1}{3}\\
e_{1}-e_{2}+e_{4}-e_{5}+e_{7} & \textrm{if} & \mymod{n}{2}{3} 
\end{array}%
\right. ,
\end{equation}
\begin{equation}\label{t10}
jO_{n}^{(3)}-4JO_{n}^{(3)}=\left\{ 
\begin{array}{ccc}
\left(
\begin{array}{c}
2-3e_{1}+e_{2}+2e_{3}\\
-3e_{4}+e_{5}+2e_{6}-3e_{7}
\end{array}%
\right)  & \textrm{if} & \mymod{n}{0}{3} \\ 
\left(
\begin{array}{c}
-3+e_{1}+2e_{2}-3e_{3}\\
+e_{4}+2e_{5}-3e_{6}+e_{7}
\end{array}%
\right) & \textrm{if} & \mymod{n}{1}{3}\\
\left(
\begin{array}{c}
1+2e_{1}-3e_{2}+e_{3}\\
+2e_{4}-3e_{5}+e_{6}+2e_{7}
\end{array}%
\right) & \textrm{if} & \mymod{n}{2}{3} 
\end{array}%
\right. .
\end{equation}
\end{theorem}
\begin{proof}
The following recurrence relation 
\begin{equation}\label{ecua1}
jO_{n+3}^{(3)}-3JO_{n+3}^{(3)}=\sum_{s=0}^{7}(j_{n+s+3}^{(3)}-3J_{n+s+3}^{(3)})e_{s}
\end{equation}
can be readily written considering that $$JO_{n}^{(3)}=J_{n}^{(3)}+\sum_{s=1}^{7}J_{n+s}^{(3)}e_{s}\ \textrm{and}\ jO_{n}^{(3)}=j_{n}^{(3)}+\sum_{s=1}^{7}j_{n+s}^{(3)}e_{s}.$$ Notice that $j_{n+3}^{(3)}-3J_{n+3}^{(3)}=2j_{n}^{(3)}$ from (\ref{e5}) (see \cite{Cook-Bac}), whence it follows that (\ref{ecua1}) can be rewritten as $jO_{n+3}^{(3)}-3JO_{n+3}^{(3)}=2jO_{n}^{(3)}$ from which the desired result (\ref{t7}) of Theorem \ref{th:3}. In a similar way we can
show the second equality. By using the identity $j_{n}^{(3)}+j_{n+1}^{(3)}=3J_{n+2}^{(3)}$ we have $$jO_{n}^{(3)}+jO_{n+1}^{(3)}=3(J_{n+2}^{(3)}e_{0}+J_{n+3}^{(3)}e_{1}+\cdots +J_{n+9}^{(3)}e_{7}),$$ which is the assertion (\ref{t8}) of theorem.

By using the identity $j_{n}^{(3)}-J_{n+2}^{(3)}=1$ from (\ref{e8}) (see \cite{Cook-Bac}) we have
\begin{align*}
jO_{n}^{(3)}-JO_{n+2}^{(3)}&=(j_{n}^{(3)}-J_{n+2}^{(3)})e_{0}+(j_{n+1}^{(3)}-JO_{n+3}^{(3)})e_{1}+\cdots +(j_{n+7}^{(3)}-J_{n+9}^{(3)})e_{7}\\
&=1-e_{1}+e_{3}-e_{4}+e_{6}-e_{7}
\end{align*}
if $\mymod{n}{0}{3}$, the other identities are clear from equation (\ref{e8}). Finally, the proof of Eq. (\ref{t10}) is similar to (\ref{t9}) by using (\ref{e6}).
\end{proof}

The following theorem investigates the multiplications $jO_{n}^{(3)}\cdot JO_{n}^{(3)}$ and $JO_{n}^{(3)}\cdot jO_{n}^{(3)}$ in the octonion algebra $\Bbb{O}$, which is a real non-commutative normed division algebra.
\begin{theorem}\label{thm:4}
Let $n\geq 0$ be integer such that $\mymod{n}{0}{3}$. Then,
\begin{equation}
\begin{aligned}
49&(jO_{n}^{(3)}\cdot JO_{n}^{(3)})\\
&=(99-520\cdot 2^{n}-349488\cdot 2^{2n})e_{0}+ (2^{2n+6}-7842\cdot 2^{n}+36)e_{1}\\
&\ \ + (2^{2n+7}+374\cdot 2^{n}-12)e_{2}+ (2^{2n+8}-4936\cdot 2^{n}-24)e_{3}\\
&\ \ + (2^{2n+9}-3390\cdot 2^{n}+36)e_{4}+ (2^{2n+10}-1110\cdot 2^{n}-12)e_{5}\\
&\ \ + (2^{2n+11}-5944\cdot 2^{n}-24)e_{6}+ (2^{2n+12}+3414\cdot 2^{n}+36)e_{7}
\end{aligned} \label{equ:8}
\end{equation}
and 
\begin{equation}
\begin{aligned}
49&(JO_{n}^{(3)}\cdot jO_{n}^{(3)})\\
&=(99-520\cdot 2^{n}-349488\cdot 2^{2n})e_{0}+ (2^{2n+6}+7838\cdot 2^{n}+36)e_{1}\\
&\ \ + (2^{2n+7}-410\cdot 2^{n}-12)e_{2}+ (2^{2n+8}+4864\cdot 2^{n}-24)e_{3}\\
&\ \ + (2^{2n+9}+3274\cdot 2^{n}+36)e_{4}+ (2^{2n+10}+850\cdot 2^{n}-12)e_{5}\\
&\ \ + (2^{2n+11}+5424\cdot 2^{n}-24)e_{6}+ (2^{2n+12}-4426\cdot 2^{n}+36)e_{7}.
\end{aligned} \label{equ:9}
\end{equation}
\end{theorem}
\begin{proof}
In view of the multiplication table \ref{table:1} and the definitions (\ref{eq:2}) and (\ref{eq:3}), we obtain \begin{align*}
jO_{n}^{(3)}\cdot JO_{n}^{(3)}&=(j_{n}^{(3)}J_{n}^{(3)}-j_{n+1}^{(3)}J_{n+1}^{(3)}-j_{n+2}^{(3)}J_{n+2}^{(3)}-j_{n+3}^{(3)}J_{n+3}^{(3)}-j_{n+4}^{(3)}J_{n+4}^{(3)}\\
&\ \ -j_{n+5}^{(3)}J_{n+5}^{(3)}-j_{n+6}^{(3)}J_{n+6}^{(3)}-j_{n+7}^{(3)}J_{n+7}^{(3)})e_{0}\\
&\ \ +(j_{n}^{(3)}J_{n+1}^{(3)}+j_{n+1}^{(3)}J_{n}^{(3)}+j_{n+2}^{(3)}J_{n+3}^{(3)}-j_{n+3}^{(3)}J_{n+2}^{(3)}+j_{n+4}^{(3)}J_{n+5}^{(3)}\\
&\ \ -j_{n+5}^{(3)}J_{n+4}^{(3)}-j_{n+6}^{(3)}J_{n+7}^{(3)}+j_{n+7}^{(3)}J_{n+6}^{(3)})e_{1}\\
&\ \ +(j_{n}^{(3)}J_{n+2}^{(3)}-j_{n+1}^{(3)}J_{n+3}^{(3)}+j_{n+2}^{(3)}J_{n}^{(3)}+j_{n+3}^{(3)}J_{n+1}^{(3)}+j_{n+4}^{(3)}J_{n+6}^{(3)}\\
&\ \ +j_{n+5}^{(3)}J_{n+7}^{(3)}-j_{n+6}^{(3)}J_{n+4}^{(3)}-j_{n+7}^{(3)}J_{n+5}^{(3)})e_{2}\\
&\ \ \vdots \\
&\ \ +(j_{n}^{(3)}J_{n+7}^{(3)}-j_{n+1}^{(3)}J_{n+6}^{(3)}+j_{n+2}^{(3)}J_{n+5}^{(3)}+j_{n+3}^{(3)}J_{n+4}^{(3)}-j_{n+4}^{(3)}J_{n+3}^{(3)}\\
&\ \ -j_{n+5}^{(3)}J_{n+2}^{(3)}+j_{n+6}^{(3)}J_{n+1}^{(3)}+j_{n+7}^{(3)}J_{n}^{(3)})e_{7}.
\end{align*}

Let $a=1+\frac{2i\sqrt{3}}{3}$ and $b=1-\frac{2i\sqrt{3}}{3}$, using Eq. (\ref{h5}) and the following formula 
\begin{align*}
49(j_{n+r}^{(3)}J_{n+s}^{(3)})&=2^{2n+4+r+s}-2^{n+r+3}(a\omega_{1}^{n+s}+b\omega_{2}^{n+s})+3\cdot 2^{n+s+1}(a\omega_{1}^{n+r}+b\omega_{2}^{n+r})\\
&\ \ -3(a^{2}\omega_{1}^{2n+r+s}+b^{2}\omega_{2}^{2n+r+s})-7(\omega_{1}^{r}\omega_{2}^{s}+\omega_{1}^{s}\omega_{2}^{r}),
\end{align*}
we get the required result (\ref{equ:8}) if $\mymod{n}{0}{3}$. In the same way, from the multiplication $JO_{n}^{(3)}\cdot jO_{n}^{(3)}$ and the formula
\begin{align*}
49(J_{n+r}^{(3)}j_{n+s}^{(3)})&=2^{2n+4+r+s}+3\cdot 2^{n+r+1}(a\omega_{1}^{n+s}+b\omega_{2}^{n+s})-2^{n+s+3}(a\omega_{1}^{n+r}+b\omega_{2}^{n+r})\\
&\ \ -3(a^{2}\omega_{1}^{2n+r+s}+b^{2}\omega_{2}^{2n+r+s})-7(\omega_{1}^{r}\omega_{2}^{s}+\omega_{1}^{s}\omega_{2}^{r}),
\end{align*}
we reach (\ref{equ:9}).
\end{proof}

Based on the Binet's formulas given in (\ref{equ:4}) and (\ref{equ:5}) for the third order Jacobsthal and third order Jacobsthal-Lucas octonions, now we give some quadratic identities for these octonions.

\begin{theorem}\label{thm:5}
For every nonnegative integer number $n$ we get
\begin{equation}\label{t11}
\left( jO_{n}^{(3)}\right) ^{2}+3JO_{n+3}^{(3)}\cdot jO_{n+3}^{(3)}=4^{n+3}\underline{\alpha}^{2}+\frac{3\cdot 2^{n+1}}{49}\left(25\underline{\alpha}\cdot \widehat{\epsilon_{n}}-31\widehat{\epsilon_{n}}\cdot \underline{\alpha}\right),
\end{equation}
where $\underline{\alpha}=\sum_{s=0}^{7}2^{s}e_{s}$ and $$\widehat{\epsilon_{n}}=\left\{ 
\begin{array}{ccc}
\begin{array}{c}
2(1+e_{3}+e_{6})-3(e_{1}+e_{4}+e_{7})+(e_{2}+e_{5})
\end{array}%
  & \textrm{if} & \mymod{n}{0}{3} \\ 
\begin{array}{c}
-3(1+e_{3}+e_{6})+(e_{1}+e_{4}+e_{7})+2(e_{2}+e_{5})
\end{array}%
 & \textrm{if} & \mymod{n}{1}{3}\\
\begin{array}{c}
(1+e_{3}+e_{6})+2(e_{1}+e_{4}+e_{7})-3(e_{2}+e_{5})
\end{array}%
 & \textrm{if} & \mymod{n}{2}{3} 
\end{array}%
\right. .$$
\end{theorem}
\begin{proof}
Let $\underline{\alpha}=\sum_{s=0}^{7}2^{s}e_{s} \in \Bbb{O}$. Using the relation in (\ref{equ:4}) and (\ref{equ:5}) for the third order Jacobsthal and third order Jacobsthal-Lucas octonions, the left side of equality (\ref{t11}) can be written as
\begin{align*}
\left( jO_{n}^{(3)}\right) ^{2}&+3JO_{n+3}^{(3)}\cdot jO_{n+3}^{(3)}\\
&=\left(\frac{1}{7}(2^{n+3}\underline{\alpha}+3\widehat{\epsilon_{n}})\right)^{2}+3\left(\frac{1}{7}(2^{n+4}\underline{\alpha}-\widehat{\epsilon_{n+3}})\right)\cdot \left(\frac{1}{7}(2^{n+6}\underline{\alpha}+3\widehat{\epsilon_{n+3}})\right)\\
&=\frac{1}{49}(2^{2n+6}\underline{\alpha}^{2}+3\cdot 2^{n+1}(\underline{\alpha}\cdot \widehat{\epsilon_{n}}+\widehat{\epsilon_{n}}\cdot \underline{\alpha})+9(\widehat{\epsilon_{n}})^{2})\\
&\ \ +\frac{3}{49}(2^{2n+10}\underline{\alpha}^{2}+2^{n+4}(3\underline{\alpha}\cdot \widehat{\epsilon_{n+3}}-4\widehat{\epsilon_{n+3}}\cdot \underline{\alpha})-3(\widehat{\epsilon_{n}})^{2})\\
&=2^{2n+6}\underline{\alpha}^{2}+\frac{3}{49}\left(2^{n+1}(\underline{\alpha}\cdot \widehat{\epsilon_{n}}+\widehat{\epsilon_{n}}\cdot \underline{\alpha})+2^{n+4}(3\underline{\alpha}\cdot \widehat{\epsilon_{n+3}}-4\widehat{\epsilon_{n+3}}\cdot \underline{\alpha})\right),
\end{align*}
where 
\begin{equation}
\begin{aligned}
\widehat{\epsilon_{n}}&=\left(1+\frac{2i\sqrt{3}}{3}\right)\underline{\omega_{1}}\omega_{1}^{n}+\left(1-\frac{2i\sqrt{3}}{3}\right)\underline{\omega_{2}}\omega_{2}^{n}\\
&=\sum_{s=0}^{7}\left(\left(1+\frac{2i\sqrt{3}}{3}\right)\omega_{1}^{n+s}+\left(1-\frac{2i\sqrt{3}}{3}\right)\omega_{2}^{n+s}\right)e_{s}\\
&=\left\{ 
\begin{array}{ccc}
\begin{array}{c}
2(1+e_{3}+e_{6})-3(e_{1}+e_{4}+e_{7})+(e_{2}+e_{5})
\end{array}%
  & \textrm{if} & \mymod{n}{0}{3} \\ 
\begin{array}{c}
-3(1+e_{3}+e_{6})+(e_{1}+e_{4}+e_{7})+2(e_{2}+e_{5})
\end{array}%
 & \textrm{if} & \mymod{n}{1}{3}\\
\begin{array}{c}
(1+e_{3}+e_{6})+2(e_{1}+e_{4}+e_{7})-3(e_{2}+e_{5})
\end{array}%
 & \textrm{if} & \mymod{n}{2}{3} 
\end{array}%
\right. .
\end{aligned}\label{t12}
\end{equation}
Note that $\widehat{\epsilon_{n}}=\widehat{\epsilon_{n+3}}$ for all $n\geq 0$, which can be simplified as
\begin{align*}
\left( jO_{n}^{(3)}\right) ^{2}&+3JO_{n+3}^{(3)}\cdot jO_{n+3}^{(3)}\\
&=2^{2n+6}\underline{\alpha}^{2}+\frac{3\cdot 2^{n+1}}{49}\left(\underline{\alpha}\cdot \widehat{\epsilon_{n}}+\widehat{\epsilon_{n}}\cdot \underline{\alpha}+24\underline{\alpha}\cdot \widehat{\epsilon_{n}}-32\widehat{\epsilon_{n}}\cdot \underline{\alpha}\right)\\
&=4^{n+3}\underline{\alpha}^{2}+\frac{3\cdot 2^{n+1}}{49}\left(25\underline{\alpha}\cdot \widehat{\epsilon_{n}}-31\widehat{\epsilon_{n}}\cdot \underline{\alpha}\right).
\end{align*}
Thus, we get the required result in (\ref{t11}).
\end{proof}

\begin{theorem}\label{thm:6}
For every nonnegative integer number $n$ we get
\begin{equation}\label{t13}
\left( jO_{n}^{(3)}\right) ^{2}-9\left( JO_{n}^{(3)}\right)^{2}=\frac{2^{n+1}}{7}(2\underline{\alpha}^{2}+3(\underline{\alpha}\cdot \widehat{\epsilon_{n}}+\widehat{\epsilon_{n}}\cdot \underline{\alpha})),
\end{equation}
where $\underline{\alpha}=\sum_{s=0}^{7}2^{s}e_{s}$ and $\widehat{\epsilon_{n}}$ as in (\ref{t12}).
\end{theorem}
The proofs of quadratic identities for the third order Jacobsthal and third order Jacobsthal-Lucas octonions in this theorem are similar to the proof of the identity (\ref{t11}) of Theorem \ref{thm:5}, and are omitted here.

\end{document}